\documentclass[letterpaper,11pt]{amsart}
\usepackage{amsmath, amsfonts, amssymb, mathrsfs}
\usepackage[colorlinks, citecolor=blue]{hyperref}

\newcommand{\nc}{\newcommand}
\nc{\dmo}{\DeclareMathOperator}
\nc{\nt}{\newtheorem}

\nt{theorem}{Theorem}[section]
\nt{lemma}[theorem]{Lemma}
\nt{proposition}[theorem]{Proposition}
\nt{corollary}[theorem]{Corollary}

\newtheorem{thm}{Theorem}

\nc{\margin}[1]{\marginpar{\tiny #1}}
\nc{\p}[1]{\smallskip\noindent{{\bf #1}}}

\dmo{\ra}{\rightarrow}

\title{Spectral radius of a star with one long arm}
\author{Hyunshik Shin}


\begin{document}
\begin{abstract}
A tree is said to be starlike if exactly one vertex has degree greater than two.
In this paper, we will study the spectral properties of $S(n,k \cdot 1)$, that is,
the starlike tree with $k$ branches of length 1 and one branch of length $n$.
The largest eigenvalue $\lambda_1$ of $S(n,k \cdot 1)$ satisfies
$\sqrt{k+1} \leq \lambda_1 < k/\sqrt{k-1}$. Moreover, the largest eigenvalue of $S(n,k \cdot 1)$
is equal to the largest eigenvalue of $S(k \cdot (n+1) )$, which is the starlike tree that has
$k$ branches of length $n-1$.
Using the spectral radii of $S(n,k \cdot 1)$ we can show that there is a sequence of
Salem numbers that converges to each integer $>1$.
\end{abstract}
\maketitle

\section{Introduction}
\label{section:introduction}

A tree which has exactly one vertex of degree greater than two is said to be \textit{starlike}.
Spectral properties of starlike trees are recently studied in \cite{LepovicGutman01, LepovicGutman02, BellSimic98}.

Let $P_n$ be the path with $n$ vertices. We denote $S(n_1, n_2, \cdots, n_k)$ a starlike
tree in which removing the central vertex $v_1$ leaves disjoint paths such that
$$S(n_1, n_2, \cdots, n_k) - v_1 = P_{n_1} \cup P_{n_2} \cup \cdots \cup P_{n_k}.$$
We say that the starlike tree $S(n_1, n_2, \cdots, n_k)$ has $k$ branches and the lengths
of branches are $n_1, n_2, \cdots, n_k$. It will be assumed that 
$n_1 \geq n_2 \geq \cdots \geq n_k$.

For a simple graph $G$ of order $n$, the \textit{spectrum} of $G$ is the set of eigenvalues
$\lambda_1 \geq \lambda_2 \geq \cdots \geq \lambda_n$ of its adjacency matrix $A$.
The characteristic polynomial $\det(\lambda I - A) $ of $A$ is called the characteristic
polynomial of $G$, denoted $\phi(G,\lambda)$ or simply $\phi(G)$. It is known that 
if $G$ is a graph and $v$ is any vertex, then
\[
\phi(G) = \lambda \phi(G-v) - \sum_{u} \phi(G-v-u) - 2 \sum_{C} \phi(G-C),
\]
where the first summation is over vertices $u$ adjacent to the vertex $v$ and
the second summation is over all cycles $C$ embracing the vertex $v$.
Applying to the starlike trees we obtain
\begin{equation}
\label{eqn:charpoly}
\phi(S(n_1, n_2, \cdots, n_k)) = \lambda \prod_{i=1}^{k} \phi(P_{n_i}) - \sum_{i=1}^k \left[ \phi(P_{n_i-1}) \prod_{j \in I_i} \phi(P_{n_j}) \right],
\end{equation}
where $I_i = \{ 1, 2, \cdots, k \} \setminus \{i\}$.

Using Equation (\ref{eqn:charpoly}), Lepovi\'c and Gutman \cite{LepovicGutman01} determine the bounds
for the largest eigenvalues of starlike trees.

\medskip
\begin{theorem}\cite[Theorem 2]{LepovicGutman01}
\label{thm:starlike_spectral}
If $\lambda_1$ is the largest eigenvalue of the starlike tree $S(n_1, n_2, \cdots, n_k)$,
then 
$$ \sqrt{k} \leq \lambda_1 < \frac{k}{\sqrt{k-1}}$$
for any positive integers $n_1 \geq n_2 \geq \cdots \geq n_k \geq 1$.
\end{theorem}

The lower bound $\sqrt{k}$ for $\lambda_1$ is realized by the star on $k+1$ vertices.
The upper bound can be achieved asymptotically by the starlike trees with 
$n_1 = n_2 = \cdots = n_k =n$. In such case, we will denote this starlike tree by $S(k \cdot n)$ instead
of $S(n_1, n_2, \cdots, n_k)$. 


\medskip
In this paper, we will discuss the spectral properties of a star with one long arm;
let $S(n, k \cdot 1)$ be the starlike tree with $k$ branches of length 1 and one branch
with length $n$. Note that $S(n, k \cdot 1)$ is a tree on $n+k+1$ vertices.

Two nonisomorphic graphs with the same spectrum are called \textit{cospectral}.
It is known that no two starlike trees are cospectral \cite{LepovicGutman02}.
However, the spectral radius does not distinguish starlike trees. We will show that
there are infinitely many pairs of nonisomorphic starlike trees that have the same spectral
radius.

\medskip
\begin{thm}
\label{thm:cospectral_radius}
For any positive integer $k \geq 3$, starlike trees $S(n, k \cdot 1)$ and $S(k \cdot (n+1))$ have the same largest eigenvalue.
\end{thm}

\medskip
By Theorem \ref{thm:starlike_spectral} the largest eigenvalue 
$\lambda_1$ of $S(n, k \cdot 1)$ satisfies
$$ \sqrt{k+1} \leq \lambda_1 < \frac{k+1}{\sqrt{k}}.$$
As a consequence of Theorem \ref{thm:cospectral_radius}, we have a sharper upper bound for
the starlike tree $S(n, k \cdot 1)$.

\medskip
\begin{thm}
\label{thm:bounds}
If $\lambda_1$ is the largest eigenvalue of $S(n, k \cdot 1)$, then
$$ \sqrt{k+1} \leq \lambda_1 < \frac{k}{\sqrt{k-1}}$$
for any positive integers $n \geq 1$ and $k \geq 3$.
\end{thm}

\medskip
There are two special algebraic integers related to the largest eigenvalue
of starlike trees.
A \textit{Salem number} is an algebraic integer $\alpha > 1$,
all of whose other conjugates have modulus $\leq 1$, with at least one
conjugate of modulus 1.
A \textit{Pisot number} is an algebraic integer $\beta > 1$,
all of whose other conjugates have modulus $<1$.
With Theorem \ref{thm:bounds} and the work of
McKee--Rowlinson--Smyth \cite{McKeeRowlinsonSmyth99}, we have the 
following corollary.

\medskip
\newtheorem*{thm:salem}{Corollary \ref{thm:salem_number}}
\begin{thm:salem}
For $n \geq 2$ and $k\geq 3$ let $\lambda_1$ be the largest eigenvalue of the starlike tree $S(n, k \cdot 1)$.
Then the number $t > 1 $ defined by 
\begin{equation}
\label{eqn:stretchfactor}
\sqrt{t} + \frac{1}{\sqrt{t}} = \lambda_1,
\end{equation}
is a Salem number.
\end{thm:salem}

Using the spectral properties of the starlike tree $S(n, k \cdot 1)$, the author studied
the stretch factors of pseudo-Anosov mapping classes of closed orientable surfaces.
In particular, the number $t$ defined by (\ref{eqn:stretchfactor}) is the stretch
factor of a pseudo-Anosov mapping class from Thurston's construction whose configuration
graph is $S(n, k \cdot 1)$. For more about this topic, see \cite{Shin15}.

\medskip
\section{Bounds for the largest eigenvalue}

In this section we will prove main theorems of this paper. 
Lepovi\'c and Gutman \cite{LepovicGutman01} show that the number $t >1$,
defined by $\sqrt{t} + 1/\sqrt{t} = \lambda_1$, where $\lambda_1$ is the largest
eigenvalue of $S(k \cdot (n+1))$, is the root of the polynomial equation
\begin{equation}
\label{eqn:salempoly}
t^{n+3} - (k-1) t^{n+2} + (k-1)t -1 =0.
\end{equation}

To prove Theorem \ref{thm:cospectral_radius} we will show that when $\lambda_1$ is the
largest eigenvalue of $S(n, k \cdot 1)$, the number $t$ given by 
$\sqrt{t} + 1/\sqrt{t} = \lambda_1$, is again the root of the polynomial 
(\ref{eqn:salempoly}).

\medskip
\begin{proof}[Proof of Theorem \ref{thm:cospectral_radius}]
Equation (\ref{eqn:charpoly}) reduces to
{\setlength\arraycolsep{2pt}
\begin{eqnarray*}
\phi\left( S(n, k \cdot 1) \right) & = & \lambda \phi(P_n) \phi(P_1)^k - \left( k \, \phi(P_n) \phi(P_1)^{k-1} + \phi(P_{n-1}) \phi(P_1)^k \right)\\
&=& \lambda^{k+1} \phi(P_n) - k \lambda^{k-1} \phi(P_n) - \lambda^k \phi(P_{n-1})\\
&=& \lambda^{k-1} \left( \lambda^2 \phi(P_n) - k \phi(P_n) - \lambda \phi(P_{n-1}) \right).
\end{eqnarray*}
}
Therefore the largest eigenvalue of $S(n, k \cdot 1)$ is the root of 
\begin{equation}
\label{eqn:proof_A1}
\lambda^2 \phi(P_n) - k \phi(P_n) - \lambda \phi(P_{n-1})=0.
\end{equation}
By substituting $\lambda = 2 \cos \theta$, we get $\phi(P_n) = \sin(n+1)\theta / \sin \theta$ (see \cite[p.73]{CvetkovicDoobSachs95})
and Equation (\ref{eqn:proof_A1}) becomes
\begin{equation}
\label{eqn:proof_A2}
(4\cos^2 \theta - k ) \frac{\sin(n+1)\theta}{\sin \theta} - 2 \cos \theta \frac{\sin n\theta}{\sin \theta} = 0.
\end{equation}
By setting $t^{1/2} = e^{i\theta}$, we have 
$$ \lambda = 2 \cos \theta = t^{1/2} + t^{-1/2} $$
and
$$ \sin n\theta = \frac{t^{n/2} - t^{-n/2}}{2i}.$$
By substituting and simplifying, Equation (\ref{eqn:proof_A2}) becomes
\begin{equation}
\label{eqn:proof_A3}
t^{n+3} - (k-1) t^{n+2} + (k-1)t -1	 =0.
\end{equation}
If $t^{*}$ is a root of Equation (\ref{eqn:proof_A3}), then the number $\lambda^{*}$,
defined by $\lambda^{*} = \sqrt{t^{*}} + 1/\sqrt{t^{*}}$, is a root of Equation
 (\ref{eqn:proof_A1}).
Since Equation (\ref{eqn:proof_A3}) is identical with Equation (\ref{eqn:salempoly})
we can conclude that the largest eigenvalue of $S(n, k \cdot 1)$ is equal to
the largest eigenvalue of $S(k \cdot (n+1))$.
\end{proof}

Theorem \ref{thm:bounds} follows directly from Theorem 
\ref{thm:cospectral_radius} and the work of Lepovi\'c and Gutman.

\begin{proof}[Proof of Theorem \ref{thm:bounds}]
A star with $k+2$ vertices is a subgraph of $S(n, k \cdot 1)$ and its
largest eigenvalue is $\sqrt{k+1}$.
By the interlacing theorem we have $\sqrt{k+1} \leq \lambda_1$.

On the other hand, Lepovi\'c and Gutman also show that Equation (\ref{eqn:proof_A3})
has a zero in the interval $(k-2,k-1)$ and it follows that $\lambda_1 < k/\sqrt{k-1}$
(See \cite[prrof of Theorem 2]{LepovicGutman01}).
Therefore we have
$$ \sqrt{k+1} \leq \lambda_1 < \frac{k}{\sqrt{k-1}}.$$
\end{proof}

\noindent
\textbf{Remark.}
In the paper of Lepovi\'c and Gutman, they study the properties of 
the polynomial
$$t^{n+2} - (k-1) t^{n+1} + (k-1)t -1$$
and one can easily see that all results are also true for Equation (\ref{eqn:proof_A3}).


\medskip
\section{Algebraic integers associated with starlike trees}

It is known that a starlike tree has at most one eigenvalue $> 2$.
We say that a starlike tree is \textit{hyperbolic} if it has exactly one
eigenvalue greater than $2$. It happens that all starlike trees are hyperbolic except
$S(n-3,1,1)$, for $n\geq 4$, $S(5,2,1), S(4,2,1), S(3,3,1), S(3,2,1), S(2,2,2), S(2,2,1),$ and $S(1,1,1,1)$ \cite[Theroem 1]{LepovicGutman01}.
Hence for $n\geq 2$ and $k\geq 3$, $S(n, k \cdot 1)$ is hyperbolic.

Let $\lambda_1$ be the largest eigenvalue of $S(n_1, n_2, \cdots, n_k)$.
If the starlike tree is hyperbolic, then the number
$t>1$, defined by $\sqrt{t} + 1/{\sqrt{t}} = \lambda_1$, is associated with
the dynamical complexity of an automorphism of an orientable surface
(for more about this topic, see \cite{Leininger04} or \cite{Shin15}).
In particular, $t$ is a special algebraic integer, characterized by the
following theorem.

\begin{theorem}\cite[Corollary 9]{McKeeRowlinsonSmyth99}
\label{thm:McKeeRowlinsonSmyth}
Let $S$ be a starlike tree whose largest eigenvalue $\lambda_1$ is not an
integer, and suppose that $S$ is hyperbolic. Then $t>1$, defined
by $\sqrt{t} + 1/{\sqrt{t}} = \lambda_1$, is a Salem number. If $\lambda_1$
is an integer then $t$ is a quadratic Pisot number.
\end{theorem}

\medskip
Now we have the following result.
\begin{corollary}
\label{thm:salem_number}
For $n \geq 2$ and $k\geq 3$ let $\lambda_1$ be the largest eigenvalue of the starlike tree $S(n, k \cdot 1)$.
Then the number $t > 1 $ defined by 
$$ \sqrt{t} + \frac{1}{\sqrt{t}} = \lambda_1,$$
is a Salem number.
\end{corollary}

\begin{proof}
By Theorem \ref{thm:bounds} we have
$$ k+1 < \lambda_1^2 < \frac{k^2}{k-1} = k+1 +\frac{1}{k-1}$$
and hence $\lambda_1$ is not an integer. 
By Theorem \ref{thm:McKeeRowlinsonSmyth}, $t$ is a Salem number.
\end{proof}

\medskip
Let $Q_n(t)$ be the polynomial in Equation (\ref{eqn:proof_A3}) and 
let $\rho \left( Q_n(t) \right)$ be the largest real root of $Q_n(t)$.
Let $m$ be any fixed 
positive integer. It is shown that for sufficiently large $n$, 
$Q_n(t)$ has a root in the interval $(k-1-\frac{1}{10^m}, k-1)$
(see the proof of Corollary 2.1. in \cite{LepovicGutman01}).
This implies that
$$ \lim_{n \rightarrow \infty} \rho \left( Q_n(t) \right) = k-1.$$
Since the largest root of $Q_n(t)$ is a Salem number for each $n$, 
there is a sequence of Salem numbers that converges to each integer greater than 1.

\medskip
\bibliographystyle{alpha}
\bibliography{starlikelongarm}
\end{document}